\documentclass[12pt]{amsart}
\usepackage{amsmath}
\usepackage{amsfonts}
\usepackage{amssymb}
\usepackage{url}
\usepackage{graphicx}
\usepackage{a4wide}
\usepackage{cite}

\theoremstyle{plain}
\newtheorem{lemma}{Lemma}[section]
\newtheorem{theorem}{Theorem}[section]

\theoremstyle{definition}
\newtheorem{remark}{Remark}[section]

\newtheorem{defn}{Definition}[section]     

\newcommand{\Z}{\mathbb{Z}}
\newcommand{\R}{\mathbb{R}}
\newcommand{\C}{\mathbb{C}}
\newcommand{\N}{\mathbb{N}}

\renewcommand{\H}{\mathcal{H}}
\newcommand{\BigOh}{\mathcal{O}}
\newcommand{\eps}{\varepsilon}

\renewcommand{\phi}{\varphi}
\renewcommand{\theta}{\vartheta}

\DeclareMathOperator{\Arg}{\mathrm{arg}}

\numberwithin{equation}{section}

\begin{document}
\title[]{On the asymptotic behaviour of the zeros of solutions of one
  functional-differential equation with rescaling}

\author[G. Derfel]{Gregory Derfel}
\address[G. D.]{Department of Mathematics and Computer Science,
Ben Gurion University of the Negev,
Beer Sheva 84105,
Israel}
\email{derfel@math.bgu.ac.il}

\author[P. J. Grabner]{Peter J. Grabner\textsuperscript{\dag{}}\textsuperscript{*}}
\thanks{\dag{} This author is supported by the Austrian Science
  Fund FWF project F5503 (part of the Special Research Program (SFB)
  ``Quasi-Monte Carlo Methods: Theory and Applications'')}
\address[P. G. \& R. T.]{Institut f\"ur Analysis und Zahlentheorie,
Technische Universit\"at Graz,
Kopernikusgasse 24,
8010 Graz,
Austria}
\email{peter.grabner@tugraz.at}

\author[R. F. Tichy]{Robert F. Tichy\textsuperscript{\ddag{}}\textsuperscript{*}}
\thanks{\ddag{} This author is supported by the Austrian Science
  Fund FWF project F5510 (part of the Special Research Program (SFB)
  ``Quasi-Monte Carlo Methods: Theory and Applications'')}
\email{tichy@tugraz.at}

\thanks{\textsuperscript{*} These authors are supported by the Doctoral
  Program ``Discrete Mathematics'' W1230.}

\dedicatory{Dedicated to Heinz Langer on the occasion of his
  80\textsuperscript{th} birthday.}

\subjclass[2000]{Primary: ; Secondary: }
\date{\today}
\begin{abstract}
We study the asymptotic behaviour of the solutions of a functional-
differential equation with rescaling, the so-called pantograph
equation. From this we derive asymptotic information about the zeros
of these  solutions.
\end{abstract}

\maketitle

\section{Introduction}\label{sec:introduction}
\subsection{Historical  remarks }

The importance of functional and functional-differential equations
with rescaling is being increasingly recognised in the last decades,
as is their relevance to a wide range of application areas.  Such
equations -- functional and functional-differential, linear and
nonlinear -- appear as adequate tools in a number of phenomena that
display some kind of self-similarity.  It is next to impossible to
describe all recent activities in this area.  For general references
and bibliography we refer the reader to survey papers by Derfel
and co-authors \cite{Derfel1995:FDE_FE_rescaling,
Derfel_Grabner_Vogl2012:laplace_fractal_Poincare,
BDM}.

One of the best known examples of equations with rescaling is the celebrated \emph{pantograph
equation}:  
\begin{equation}\label{eq:pantograph}
  y'(z)=ay(\lambda z)+by(z).
\end{equation}
This equation  was  introduced by Ockendon \& Tayler \cite{OT} as a
mathematical model of the overhead current collection system on an
electric locomotive.
(The term `pantograph equation' was coined by Iserles~\cite{Iserles}.)
This equation and its ramifications have emerged in a striking range
of applications, including number theory \cite{Mahler}, astrophysics
\cite{Chandra}, queues \& risk theory \cite{Gaver}, stochastic games
\cite{Ferguson}, quantum theory \cite{Spiridonov},  population
dynamics \cite{HallWake1}, and  graph theory \cite{Robinson1973:counting_acyclic_digraphs}. 
The common feature of all such examples
is some self-similarity of the system under study.

In 1972, Morris, Feldstein and Bowen \cite{Morris_Feldstein_Bowen1972: Phragmen_Lindelof_FDE}
studied functional-differential equations  (FDE)  of the form:
\begin{equation}\label{eq:multi1}
    y'(z)=\sum_{k=1}^\ell a_ky(\lambda_k z)
\end{equation}
with $1>\lambda_\ell>\lambda_{\ell-1}>\cdots>\lambda_1>0$,
$a_1,\ldots, a_\ell \in\C$. They were able to obtaine deep results about the
existence, uniqueness, and asymptotic behaviour of solutions of
\eqref{eq:multi1} in the complex plane $\C$.  (For more about FDE in the
complex plane see also
\cite{Derfel_Iserles1997:pantograph_complex_domain}.)

Among other things,  Morris \emph{et al.} gave a detailed analysis of
\eqref{eq:pantograph}  in the special case $a=-1, b=0 $, i.e.,
\begin{equation}\label{eq:pantograph_b=0}
  y'(z)= -y(\lambda z)
\end{equation}
In particular,  they proved that 
\eqref{eq:pantograph},   supplemented by the initial condition 
\begin{equation}\label{eq:initial_condition}
y(0)=1
\end{equation}
has the unique solution 
\begin{equation}\label{eq:tailor_reprezentation}
  y(z)=\sum_{n=0}^\infty (-1)^n\lambda^\frac{n(n-1)}{2}\frac{z^n}{n!}.
\end{equation}
Moreover,  $y(z)$ is an  entire function of order zero, and
it  has infinitely many positive zeros, but no other zeros in the complex
plane.  
The entire function \eqref{eq:tailor_reprezentation} is sometimes
called   the  \emph{deformed exponential function} (see Sokal \cite
{Sokal2012:leading_root_theta_function}).

A number of conjectures on the zeros $0<t_0<t_1<t_2\ldots$ of
\eqref{eq:tailor_reprezentation} have been made by Morris, Feldstein
and Bowen \cite{Morris_Feldstein_Bowen1972: Phragmen_Lindelof_FDE} and
by Iserles \cite{Iserles}. In particular, Morris, Feldstein and Bowen
conjectured that
\begin{equation} \label{eq:morris_feldstein_conjecture}
 \lim_{n\rightarrow\infty} t_{n+1}/t_{n}=1/\lambda :=q
\end{equation}
Also,  in what follows we shall use the notation  $q:= 1/\lambda>1$. 

It is notable  that in  1973, independently of
\cite{Morris_Feldstein_Bowen1972: Phragmen_Lindelof_FDE}, Robinson, in his paper on counting
of acyclic digraphs \cite{Robinson1973:counting_acyclic_digraphs}, derived the
same FDE as \eqref{eq:pantograph_b=0} and  conjectured 
\begin{equation} \label{eq:robinson_conjecture}
t_{n}=(n+1)q^n + o(q^n),
\end{equation}
Clearly, \eqref{eq:robinson_conjecture} is  stronger than \eqref{eq:morris_feldstein_conjecture}.
In 2000, 
Langley   \cite{Langley2000:functional-differential}   resolved  Morris's \emph{et al.} conjecture by proving 
that
\begin{equation}\label{eq:langley}
t_{n} = nq^{n-1}(\gamma + o(1)),
\end{equation}
where $\gamma$ is a positive constant.  In 2005,   Grabner and Steinsky  \cite{Grabner_Steinsky2005:asymptotic_behaviour_poles},
independently  of  \cite{Langley2000:functional-differential},  proved  a weaker form of Robinson's conjecture: namely,
 there exists $k_0$ such that
\begin{equation}\label{eq:grabner_steinsky}
t_{k_0 +k}=(k+1)q^k + o(q^k/{k^{(1-\varepsilon}}),
\end{equation}
for all $\varepsilon>0$.

Recently, Zhang  \cite{Zhang2016:functional-differential} proved that $\gamma=1$  
in \eqref{eq:langley}   and, moreover, 
\begin{equation}\label{eq:zhang2016}
t_{n} = nq^{n-1}(1+\psi(q)n^{-2}+o(n^{-2})),
\end{equation}
where  $\psi(q)$  is the generating function of the sum-of-divisors function $\sigma(k)$
\cite{Langley2000:functional-differential}.
Also, he derived an asymptotic formula for  the oscillation amplitude $A_n$ of  $y(x)$, i.e.,
$A_n=|y(qt_n)|$.

\subsection{Main results }

All aforementioned results were concerned with the analytic function
\eqref{eq:tailor_reprezentation}, which is the unique solution of the
Cauchy problem \eqref{eq:pantograph_b=0}--\eqref{eq:initial_condition}. 
In contrast to this, in the present
paper we deal with \emph{all solutions} of \eqref{eq:pantograph_b=0},
not necessarily also satisfying \eqref{eq:initial_condition}, but
rather \emph{defined on an arbitrary half-line}.

The following natural definition  is commonly accepted
in the theory of functional-differential equations:

\begin{defn}
  If $x_0$ is a real number, then a real or complex function $y(x)$, 
  defined and continuous for $x\geq\lambda x_0$,  is said to be a
  solution of \eqref{eq:pantograph_b=0} for $x\geq x_0$, if it
  satisfies \eqref{eq:pantograph_b=0} for all $x\geq x_0$.
\end {defn}
Thus, instead of the Cauchy problem for ODE, we have an \emph {initial
  value} problem for FDE of \emph{retarded type} (i.e., $ 0<
\lambda<1$): for an arbitrary \emph {initial function} $\varphi$ defined
on $[\lambda x_0, x_0]$, a solution of the initial value problem is a
function $y(x$) defined and continuous for $x\geq\lambda x_0$, which
satisfies  \eqref{eq:pantograph_b=0} and the initial condition:
\begin{equation}
 y(x)=\varphi(x),    \qquad x\in [\lambda x_0,  x_0].
\end{equation}
Thus, the general solution of FDEs normally consists of an
\emph{infinite family of solutions}, depending on an \emph{arbitrary
  function}.

Evidently, any function that is continuous in $[\lambda x_0, x_0]$ can
be continued uniquely to a solution for $x\geq x_0$, and it makes
sense to discuss the asymptotic behaviour of all solutions  as $x
\rightarrow\infty$.

The aim of this paper is to show that the above stated asymptotic behaviour of zeros is true
not only for the analytic solution  \eqref{eq:tailor_reprezentation}, 
but for \emph{all} solutions of  \eqref{eq:pantograph_b=0}.

Namely, in Section \ref{sec:zeros},  below, we prove  that for  every solution  $y(x)$ 
of  \eqref{eq:pantograph_b=0}, 
the following asymptotic formula  for zeros $t_n$   is valid:

\begin{equation}\label{eq:DGT_zeros}
t_{n} = nq^{n-1}\left(\gamma +  \BigOh(\frac{\log k}{k})\right)
\end{equation}

It is worth noting that all solutions of \eqref{eq:pantograph_b=0},
except for \eqref{eq:tailor_reprezentation}, are \emph{non-analytic}
ones at $z=0$, but \eqref{eq:DGT_zeros} remains true for all of them.
 
In summary, we can say that the asymptotic behaviour \eqref{eq:DGT_zeros}
is an \emph {intrinsic property of the equation}
\eqref{eq:pantograph_b=0} itself, and only the constant $\gamma$ in
\eqref{eq:DGT_zeros} depends on the specific solution.
 
Our second objective (see Section
\ref{asymptotics_analytical_solitions}) is an asymptotic analysis of the
solutions of \eqref{eq:pantograph_b=0}.
 
In Section \ref{sec:case-b=0_first_order} we derive an asymptotic
formula of de Bruijn, Kato and McLeod type by a method different from
\cite{Bruijn1953:difference-differential} and
\cite{Kato_McLeod1971:functional-differential}.
 
Furthermore, in Section \ref{sec:FDE_higher_order_compressed} we give a
brief summary of related results from
\cite{Derfel1978:asymptotics_class_FDE_Kiev} and
\cite{Derfel1995:functional_differential_compressed_arguments} for
\emph{general FDE of higher order with several scaling factors}
 \begin{equation}
 y^{(m)} (x)= \sum_{j=0}^{\ell}\sum_{k=0}^{m-1}a_{jk} y^{(k)} (\alpha_j t+\beta_j),
\end{equation}
where all $\alpha_j$  are less than $1$  in modulus  i.e. $|\alpha_j|<1$.

\section{Asymptotic behaviour of analytic solutions}\label{asymptotics_analytical_solitions}
\subsection{ Asymptotic behaviour of  the solution of FDE $y'(z)=-y(\lambda z)$}\label{sec:case-b=0_first_order}
We first observe that the equation
\begin{equation}\label{eq:f}
  y'(z)=ay(\lambda z)
\end{equation}
can be simplified to
\begin{equation}\label{eq:g}
  g'(z)=-g(\lambda z);
\end{equation}
i.e. every solution of \eqref{eq:f} can be written as $y(z)=g(-az)$ for a
solution $g$ of \eqref{eq:g}.

We start with an ansatz as a power series
\begin{equation*}
  g(z)=\sum_{n=0}^\infty g_nz^n,
\end{equation*}
which gives the recursion
\begin{equation*}
 (n+1)g_{n+1}=-\lambda^ng_n.
\end{equation*}
From this we obtain
\begin{equation}
  \label{eq:gz1}
  g(z)=\sum_{n=0}^\infty (-1)^n\lambda^{\binom n2}\frac{z^n}{n!},
\end{equation}
if we assume $g_0=1$.

In order to study the asymptotic behaviour of $g(z)$ for large $z$, we
transform \eqref{eq:gz1} into an integral representation inspired by
the inversion formula for the Mellin transform
\begin{align}
  \label{eq:gz2}
  g(z)&=\frac1{2\pi i}\int_{\H}\Gamma(s)\lambda^{\binom{-s}2}z^{-s}\,ds=
\frac1{2\pi i}\int_{\H}\frac{\pi}{\sin(\pi s)}
\frac{\lambda^{\binom{-s}2}z^{-s}}{\Gamma(1-s)}\,ds,\\
g(-z)&=\frac1{2\pi i}\int_{\H}\pi\cot(\pi s)
\frac{\lambda^{\binom{-s}2}z^{-s}}{\Gamma(1-s)}\,ds,\label{eq:gz3}
\end{align}
where $\H$ is a contour encircling the negative real axis counterclockwise, like
the Hankel-contour used in the theory of the $\Gamma$-function (see
Figure~\ref{fig:hankel}). We will use the representation \eqref{eq:gz2} for
deriving an asymptotic formula for $g(z)$ for $|\Arg(z)|\leq\pi-\eps$ (for
$\eps>0$), whereas \eqref{eq:gz3} will be used for the asymptotic of $g(z)$ for
$|\Arg(-z)|\leq\pi-\eps$. These representations can be proved by residue
calculus and taking care of the growth order of the integrand along the
contour. In the case $\lambda=\frac12$ a similar representation was used in
\cite{Grabner_Steinsky2005:asymptotic_behaviour_poles}.
\begin{figure}[h]
 \centering

 \includegraphics[width=\textwidth]{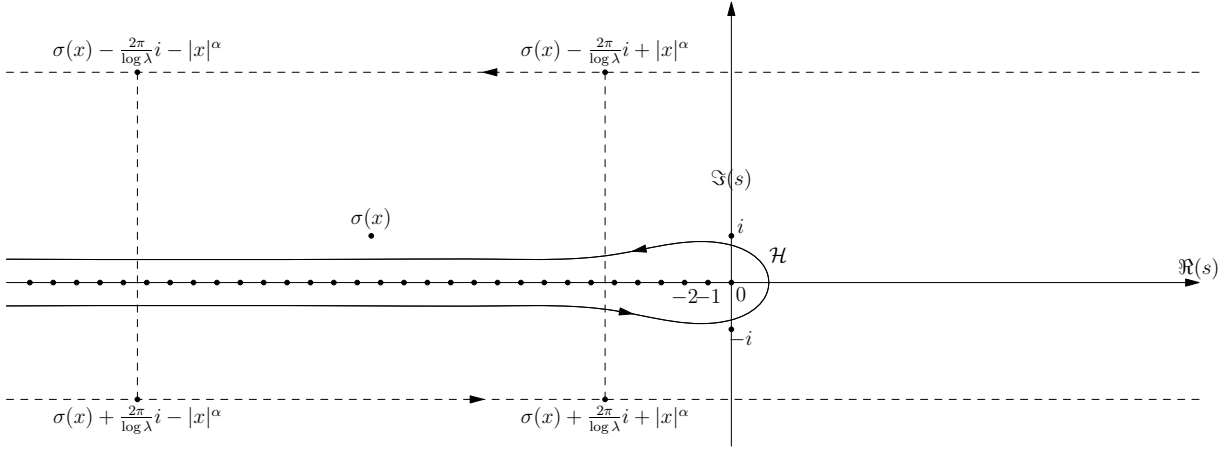}
 \caption{The contour of integration $\H$}
 \label{fig:hankel}
\end{figure}

Similarly, the representation
\begin{equation*}
  g(z)=\frac1{2\pi i}
\left(\int\limits_{-\infty-i}^{\infty-i}-\int\limits_{-\infty+i}^{\infty+i}\right)
\Gamma(s)\lambda^{\binom{-s}2}z^{-s}\,ds\\
\end{equation*}
can be shown. Here the contour of integration is deformed into two horizontal
lines above and below the real axis.

\begin{theorem}\label{thm:asy1}
  Let $g$ be the entire solution of \eqref{eq:g} with $g(0)=1$. Then for
  $\eps>0$ and $|\Arg(z)|\leq\pi-\eps$ and $z\to\infty$ the asymptotic expansion
  \begin{multline}
    \label{eq:asy1}
    g(z)\sim Cz^A \log(z)^B\exp\left(-\frac1{2\log\lambda}
\left(\log(z)-\log(\log(z))\right)^2\right)\\
\times H\left(\frac1{\log\lambda}\left(\log(z)-\log(\log(z))
\right)-\frac12+\frac{\log(-\log\lambda)}{\log\lambda}\right)
  \end{multline}
  holds, where the periodic function $H$ of period $2$ is given by the Fourier
  series
\begin{equation}
  \label{eq:h-fourier}
  H(x)=\sqrt{\frac{2\pi}{-\log\lambda}}
\sum_{k\in\mathbb{Z}}e^{\frac{(2k+1)^2\pi^2}{2\log\lambda}}e^{\pi i(2k+1)x},
\end{equation}
and
\begin{equation}\label{eq:ABC}
  \begin{aligned}
    A&=\frac12-\frac1{\log\lambda}-\frac{\log(-\log\lambda)}{\log\lambda}\\
    B&=\frac{\log(-\log\lambda)}{\log\lambda}-1\\
    C&=\exp\left(\frac12-\frac{\log\lambda}8+\log(-\log\lambda)-
      \frac{\log(-\log\lambda)}{\log\lambda}-
      \frac{(\log(-\log\lambda))^2}{2\log\lambda}-\frac12\log(2\pi)\right).
  \end{aligned}
\end{equation}
\end{theorem}
\begin{remark}
  This theorem extends the real asymptotic of \textbf{all} solutions on
  $\mathbb{R}^+$ given in \cite{Kato_McLeod1971:functional-differential} to the
  complex asymptotic of the \textbf{entire} solution in an angular region
  avoiding the negative real axis.
\end{remark}
\begin{theorem}\label{thm:asy2}
  Let $g$ be the entire solution of \eqref{eq:g} with $g(0)=1$. Then for
  $\eps>0$ and $|\Arg(-z)|\leq\pi-\eps$ and $z\to\infty$ the asymptotic
  expansion
  \begin{multline}
    \label{eq:asy2}
    g(-z)\sim Cz^A \log(-z)^B\exp\left(-\frac1{2\log\lambda}
\left(\log(-z)-\log(\log(-z))\right)^2\right)\\
\times K\left(\frac1{\log\lambda}\left(\log(-z)-\log(\log(-z))\right)
-\frac12+\frac{\log(-\log\lambda)}{\log\lambda}\right)
  \end{multline}
  holds, where the periodic function $K$ of period $1$ is given by the Fourier
  series
\begin{equation}
  \label{eq:k-fourier}
  K(x)=\sqrt{\frac{2\pi}{-\log\lambda}}\sum_{k\in\mathbb{Z}}
e^{\frac{2k^2\pi^2}{\log\lambda}}e^{2\pi ikx},
\end{equation}
and $A$, $B$, and $C$ are given by \eqref{eq:ABC}.
\end{theorem}
\begin{proof}[Proof of Theorem~\ref{thm:asy1}]
  We apply a saddle point approximation combined with residue calculus to the
  second integral representation given in \eqref{eq:gz2}. For this purpose, we
  consider
  \begin{equation}\label{eq:kernel}
    \frac{\lambda^{\binom{-s}2}z^{-s}}{\Gamma(1-s)}=
  \exp\left(\frac{s(s+1)}2\log\lambda-s\log(z)-\log(-s)-\log(\Gamma(-s))\right).
  \end{equation}
We find the saddle point as the stationary point of the argument of the
exponential:
\begin{equation}\label{eq:saddle}
  \left(s+\frac12\right)\log\lambda-x-\frac1s+\psi(-s)=0,
\end{equation}
where $\psi=\frac{\Gamma'}{\Gamma}$; for simplicity, we set $x=\log(z)$. From
\cite{Magnus_Oberhettinger_Soni1966:formulas_theorems_special} we infer
\begin{equation*}
  \psi(-s)=\log(-s)+\frac1{2s}+\BigOh\left(\frac1{s^2}\right),
\end{equation*}
valid for $|\Arg(-s)|\leq\pi-\eps$ for $\eps>0$.
This shows that \eqref{eq:saddle} has a unique solution $\sigma(x)$ satisfying
\begin{equation}\label{eq:sigma}
\sigma(x)=\frac1{\log\lambda}\left(x-\log(x)\right)-\frac12+
\frac{\log(-\log\lambda)}{\log\lambda}+\BigOh(\log(x)/x),
\end{equation}
which lies close to the negative real axis; notice that the imaginary part of
$\sigma(x)$ is bounded by $\frac\pi{|\log\lambda|}$. Around $s=\sigma(x)$ we
have the following approximation
\begin{multline}\label{eq:saddle1}
  \frac{s(s+1)}2\log\lambda-sx-\log(s)-\log(\Gamma(-s))\\=
\frac{\sigma(x)(\sigma(x)+1)}2\log\lambda-x\sigma(x)-\log(-\sigma(x))-
\log(\Gamma(-\sigma(x)))\\
+\frac12\left(\log\lambda+\frac1{\sigma(x)^2}-\psi'(-\sigma(x))\right)
\left(s-\sigma(x)\right)^2+\BigOh\left(x^{-2}(s-\sigma(x))^3\right)
\end{multline}
for $s-\sigma(x)=\BigOh(x^\alpha)$ for $\alpha<\frac23$. 

Inserting the approximation \eqref{eq:saddle1} into the integral representation
\eqref{eq:gz2} (and splitting the range of integration into $|s-\sigma(x)|\leq
x^\alpha$ and $|s-\sigma(x)|\geq x^\alpha$) yields
\begin{multline}\label{eq:gz4}
  g(z)=\frac{\lambda^{\binom{-\sigma(x)}2}}{\Gamma(1-\sigma(x))}e^{-x\sigma(x)}
\frac1{2\pi i}\oint_{R_x}\exp\left(\frac12(s-\sigma(x))^2\log\lambda\right)
\frac\pi{\sin(\pi s)}\,ds
\left(1+\BigOh(x^{2\alpha-1})\right)\\+\frac1{2\pi
  i}\int_{R'_x}\Gamma(s)\lambda^{\binom{-s}2} z^{-s}\,ds,
\end{multline}
if we choose $\alpha<\frac12$. Here $R_x$ denotes the positively oriented
rectangle with corners $\sigma(x)\pm |x|^\alpha\pm\frac{2\pi i}{\log\lambda}$
and $R'_x$ denotes the remaining part of the dashed contour in
Figure~\ref{fig:hankel}. 

The first integral in \eqref{eq:gz4} can be evaluated by residue calculus
\begin{multline}\label{eq:theta}
  \frac1{2\pi i}\oint_{R_x}\exp\left(\frac12(s-\sigma(x))^2\log\lambda\right)
\frac\pi{\sin(\pi s)}\,ds\\=\sum_{\substack{n\in\Z\\|n-\sigma(x)|<|x|^\alpha}}(-1)^n
e^{\frac12(n-\sigma(x))^2\log\lambda}=\sum_{n\in\Z}(-1)^n
e^{\frac12(n-\sigma(x))^2\log\lambda}+
\BigOh\left(e^{\frac12|x|^{2\alpha}\log\lambda}\right).
\end{multline}
This series represents a continuous periodic function of period $2$, whose
Fourier coefficients can be computed by a variant of Poisson's summation formula
\begin{equation}
  \label{eq:H(y)}
  H(y)=\sum_{n\in\Z}(-1)^n
e^{\frac12(n-y)^2\log\lambda}=\frac{\sqrt{2\pi}}{\sqrt{-\log\lambda}}
\sum_{k\in\Z}e^{\frac{(2k+1)^2\pi^2}{2\log\lambda}}e^{i\pi(2k+1)y}.
\end{equation}

For estimating the remaining integral over $R_x'$ in \eqref{eq:gz4}, we use the
estimates (cf.~\cite{Magnus_Oberhettinger_Soni1966:formulas_theorems_special})
\begin{align*}
  |\Gamma(t\pm iC)|&\leq \sqrt{2\pi}(t^2+C^2)^{\frac12(t-\frac12)}
e^{-\frac{\pi C}2+\frac1{6C}}\\
|\Gamma(-t\pm iC)|&\leq\frac{\pi}{|C|\cosh(\pi C)^{\frac12}\Gamma(t)}
\end{align*}
valid for $t\geq1$ and $C>0$.

Inserting the asymptotic information about $\sigma(x)$ from \eqref{eq:sigma}
into \eqref{eq:kernel} and using Stirling's formula for the $\Gamma$-function
yields
\begin{equation*}
  \frac{\lambda^{\binom{-\sigma(x)}2}e^{-x\sigma(x)}}{\Gamma(1-\sigma(x))}=
Ce^{Ax}x^B\exp\left(-\frac1{2\log\lambda}(x-\log(x))^2\right),
\end{equation*}
with $A$, $B$, and $C$ as given in \eqref{eq:ABC}.
\end{proof}
\begin{proof}[Proof of Theorem~\ref{thm:asy2}] For the proof of \eqref{eq:asy2}
  we use the integral representation \eqref{eq:gz3} and argue along the same
  lines as in the proof of Theorem~\ref{thm:asy1}. The only technical difference
  is that the residues of $\pi\cot(\pi s)$ are all equal to $1$, which avoids
  the sign change occurring in \eqref{eq:theta}.
\end{proof}
\subsection{FDE of higher order with compressed arguments }\label{sec:FDE_higher_order_compressed}

In this section we consider the general FDE with  rescaling
\begin{equation}\label{eq:high_order}
 y^{(m)} (x)= \sum_{j=0}^{l}\sum_{k=0}^{m-1}a_{jk} y^{(k)} (\alpha_j t+\beta_j),
\end{equation}
where $a_{jk}\in\C$ and $\alpha_j, \beta_j \in \R$.
From here on we consider,
 solutions of \eqref{eq:high_order}
defined on the whole real line $\R$, when there exists at least one
$\beta_j \neq 0$, and possibly defined on a half-line $\R_+$, or
$\R_-$, when all $\beta_j=0$.

For such equations, it is hard to expect the existence of an
asymptotic formula similar to \eqref{eq:asy1}. However, we can derive
sharp estimates from above and below for solutions of
\eqref{eq:high_order}.

Below, we give a brief summary of the related results from
\cite{Derfel1978:asymptotics_class_FDE_Kiev} and
\cite{Derfel1995:functional_differential_compressed_arguments}.
    
 Denote: 
 \begin{equation}
   \alpha=\min_{0\leq j\leq l}| \alpha_j|, \qquad      A=\max_{0\leq j\leq l}| \alpha_j|.
 \end{equation}
 
 Then:
 \begin{enumerate}
 \item[(i)] Every solution $y(x)$ of \eqref{eq:high_order} is an
   analytic function, that can be extended as an entire function
   $y(z)$ of order zero in $\C$.  
 \item[(ii)] Every solution of
   \eqref{eq:high_order} satisfies the estimate
   \begin{equation}\label{eq:growth_estimate}
     |y(z)|\leq C \exp \{\gamma \log^{2}(1+|z|)\},   \qquad z\in \C, 
   \end{equation}
   for some $C>0$, and
   \begin{equation}
     \gamma>m/(2|\log A|) 
   \end{equation}
 \item[(iii)] Every solution of \eqref{eq:high_order} is unbounded
 \emph{on any ray}
   emanating from the origin $z=0$ (as it is an entire function of
   order zero). In particular, every solution is unbounded on $\R_+$
   and $\R_-$. 
   
   This result cannot be strengthened in general due to
   the existence of polynomial solutions. It was proved in
   \cite{Derfel1978:asymptotics_class_FDE_Kiev} and
   \cite{Derfel1995:functional_differential_compressed_arguments} that:
   
 \item[(iv)] A necessary and sufficient condition for the existence of
   polynomial solutions of \eqref{eq:high_order} is that:
   \begin{equation}
     \sum_{j=0}^{l} a_{j0}\alpha^n_j =0
   \end{equation}
   for some $n\in\N$.
   
 Under the assumption that \eqref{eq:high_order} has no polynomial
 solutions, and $\beta_j=0$ for all $j$, one can prove a result
 stronger than (iii).  Roughly speaking, every nontrivial solution of
 \eqref{eq:high_order} grows as $|z|\rightarrow\infty$ faster than
 $\exp \{\gamma \log^{2}|z|\}$ for some $\gamma>0$.  More precisely:
 
 \item[(v)] Every solution of \eqref{eq:high_order}, which
 \emph   {at least on one ray}  emanating from the origin  
 satisfies
 \eqref{eq:growth_estimate} for some $C>0$, and
\begin{equation}
 \gamma<1/(2|\log \alpha|) 
\end{equation}
vanishes identically.
\end{enumerate}
\begin{remark}
 The  above stated results,  formulated for the equation   \eqref{eq:high_order},
 with \emph{constant coefficients},
  remain true    with minor changes also for FDE of higher order, with several compressed arguments
  and  \emph{polynomial coefficients}:
 \begin{equation}\label{eq:high_order_polynomial}
 y^{(m)} (x)= \sum_{j=0}^{l}\sum_{k=0}^{m-1} \sum_{nu=0}^{r}  a_{jk\nu}x^{\nu} y^{(k)} (\alpha_j t+\beta_j),
\end{equation}

\end{remark}

\section{Asymptotic behaviour of the zeros}\label{sec:zeros}

The central result of this section is the following theorem.

 \begin{theorem}\label{thm: zeros}
Let $y(x)$  be a solution  of  \eqref{eq:pantograph_b=0}, 
and $0<x_0<x_1<x_2<\ldots$  be the zeros of $y(x)$. 
Then there exists a  positive constant $\gamma$, such that:
\begin{equation}\label{eq:DGT_zeros_1}
x_{n} = nq^{n-1}\left(\gamma + \BigOh(\frac{\log k}{k})\right).
\end{equation}
\end{theorem}

The proof is based on a result of  Kato \& McLeod \cite{Kato_McLeod1971:functional-differential}
and de Bruijn \cite{Bruijn1953:difference-differential}
on the asymptotic behaviour of \emph{all solutions}   of  the equation

\begin{equation}\label{eq:general_pantograph_b=0}
  y'(z)=ay(\lambda z)
\end{equation}
and  Lemma \ref{lemma: q_periodic_zero }   below.

According to  Kato and McLeod \cite[Theorem 7(iii)]{Kato_McLeod1971:functional-differential}
and  de Bruijn \cite[Sections 1.3-1.4]{Bruijn1953:difference-differential}
every solution of \eqref{eq:general_pantograph_b=0} has the following asymptotic behaviour
 
   \begin{multline}
    \label{eq:asymptotics_Kato_Mcleod1}
    y(x) = x^{A_1} \log(x)^{B_1}\exp\left(-\frac1{2\log\lambda}
\left(\log(x)-\log(\log(x))\right)^2\right)\\
\times\left( h\left(\frac1{\log\lambda}\left(\log(x)-\log(\log(x))
\right)\right) + o(1)\right),
\end{multline}
where 

\begin{equation}\label{eq:A1_B1}
  \begin{aligned}
    A_1&=\frac12-\frac1{\log\lambda}-\frac{\log(-a\log\lambda)}{\log\lambda},\\
    B_1&=\frac{\log(-a\log\lambda)}{\log\lambda}-1,
     \end{aligned}
\end{equation}
and   $h(x)$ is a periodic function of period $\log q=|\log\lambda|$, 
with some additional assumptions on its Fourier  coefficients  
(see  \cite[(6.2)--(6.3)]{Kato_McLeod1971:functional-differential}).

\begin{remark}
  Notice that this asymptotic behaviour is in accordance with the
  behaviour stated in Theorems~\ref{eq:asy1} and~\ref{thm:asy2}. If
  the value $-a\log\lambda$ is negative, the complex values of the
  power $(x/\log(x))^{-\log(-a\log\lambda)/\log\lambda}$ are
  compensated by complex values of the periodic function $h(x)$ and a
  doubling of the period to accommodate the sign change. This is the
  explanation of the fact that the periodic function in
  Theorem~\ref{eq:asy1} has period $2$.
\end{remark}

\begin{lemma}\label{lemma: q_periodic_zero }
Let $G$ be a periodic function  of period $2$,  $x_0$  the minimal positive zero of $G$,
and all zeros of $G$  located  at the points: 
$x_0 + k$; \quad $ k=0, 1, 2, \ldots$. 
Then  the zeros of the function $F(x):=G\left(\frac1{\log\lambda}\left(\log(x)-\log(\log(x)\right)\right)$.
have the following asymptotic behaviour

\begin{equation}\label{eq:q_periodic_zero}
x_{n} = nq^{n-1}\left(\gamma +  \BigOh(\frac{\log k}{k})\right)
\end{equation}
\end{lemma}

 

\begin{proof}[Proof of Lemma~\ref{lemma: q_periodic_zero }]

Let us observe first that  the  zeros of $F(x)$ are located at the points $x$, such that
\begin{equation}\label{eq:log_loglog}
 \log x - \log x\log x =(x_0 + k)\log q
\end{equation}
We   shall seek solutions of \eqref {eq:log_loglog} of the form 
\begin{equation}\label{eq:C(k)_form}
 x_k=x(k)=C(k)kq^k,
\end{equation}  
where $ C(k) $ is an unknown function of no more than   power growth, i.e.,
there exists $\alpha >0 $   such that $ C(k)= \BigOh ( k^\alpha )$.
To prove Lemma \ref{lemma: q_periodic_zero } it is enough to show that
\begin{equation}\label{eq:C(k)_final}
C(k)=\gamma+\BigOh \left(\frac{\log k}{k}\right ), 
\end{equation}
where $\gamma$  is positive constant. For that, substitute
\eqref {eq:C(k)_form}  in   \eqref {eq:log_loglog}. 
It follows from \eqref {eq:C(k)_form} that
\begin{equation}\label{eq:logx}
 \log x =k\log q \left(1+ \frac{\log k}{k \log q} + \frac{\log C(k)}{k \log q} \right),
\end{equation}
and
\begin{equation}\label{eq:loglogx}
 \log \log x=\log k + \frac{\log k}{k \log q} + \frac{\log C(k)}{k \log q} 
 + \log \log q +o \left(\frac{\log k}{k \log q}\right),
\end{equation}
Combining  \eqref{eq:log_loglog},  \eqref{eq:logx} and  \eqref{eq:loglogx}, 
after some elementary calculations we obtain that
\begin{equation}\label{eq:logx_loglogx_q_periodic_zeros}
 \left(1-\frac{1}{k\log q}\right)\log C(k)= x_0\log q + \log \log q 
 + \frac{\log k}{k\log q} +o\left(\frac{\log k}{k\log q}\right)
\end{equation}
Next, let $k\rightarrow\infty$, then from  \eqref{eq:logx_loglogx_q_periodic_zeros}
we obtain
\begin{equation}\label{eq:logC(k)}
\log C(k)=(x_0\log q + \log \log q) +\BigOh \left(\frac{\log k}{k}\right),
\end{equation}
or
\begin{equation}\label{eq:C(k)}
 C(k)=q^{x_0}\log q\left(1+  \BigOh(\frac{\log k}{k})\right).
\end{equation} 
Finally, \eqref{eq:C(k)} implies \eqref{eq:C(k)_final} and
\eqref{eq:q_periodic_zero}.
\end{proof}

\begin{proof}[Proof of Theorem ~\ref{thm: zeros}]
  First, we apply the asymptotic formula
  \eqref{eq:asymptotics_Kato_Mcleod1} to the solutions of
  \eqref{eq:pantograph_b=0}.  Observe that in this case $a=-1$ and
  $\log \lambda = -\log q<0$, and therefore the expression $ \log
  (-a\log \lambda) = \log (\log \lambda)$ in \eqref{eq:A1_B1} is a
  complex number.
  
  Having this in mind,  we can
  rewrite \eqref{eq:asymptotics_Kato_Mcleod1} and \eqref{eq:A1_B1} in
  the form:

\begin{multline}\label{eq:asymptotics_Kato_Mcleod2}
 y(x) = x^{A_2} \log(x)^{B_2}\exp\left(\frac1{2\log\ q}
\left(\log(x)-\log(\log(x))\right)^2\right)\\
\times \left(\cos \frac{\pi}{\log q}\left(\log(x)-\log(\log(x)\right)
     + i(\sin \frac{\pi}{\log q}\left(\log(x)-\log(\log(x)\right)\right)\\
\times\left(h\left(\frac1{\log q}\left(\log(x)-\log(\log(x))
\right)\right) + o(1)\right),
\end{multline}
where

\begin{equation}\label{eq:A2_B2}
  \begin{aligned}
    A_2&=\frac12 +\frac1{\log q}+\frac{\log(\log\ q)}{\log q}\\
    B_2&=  -1 - \frac{\log(\log q)}{\log q},
     \end{aligned}
\end{equation}

Both the real part and the imaginary parts of
\eqref{eq:asymptotics_Kato_Mcleod2} provide asymptotic formulas for
solutions of \eqref{eq:pantograph_b=0}, and the assumptions of Lemma
\ref{lemma: q_periodic_zero } are fulfilled.

Then, in accordance  with  Lemma  \ref{lemma: q_periodic_zero },  there exists $k_0$  
such that for all $k=0, 1, 2,\ldots$

\begin{equation}\label{eq:q_periodic_zero_fromk0}
x_{k_0 +k} = kq^{k-1}\left(\gamma_1 +  \BigOh(\frac{\log k}{k})\right),
\end{equation}
Now   \eqref{eq:C(k)_final} follows   from 
\eqref{eq:q_periodic_zero_fromk0},
 and from \eqref{eq:C(k)_final} finally we
obtain \eqref{eq:DGT_zeros_1} .
\end{proof}

\section{
The case $b\neq0$}\label{sec:case-bneq0}
The behaviour of the entire solution of \eqref{eq:pantograph} changes
considerably, if $b\neq0$. This is quite obvious from the observation
that the equation can be viewed as a perturbed differential equation
$y'=by$. Thus we expect an asymptotic behaviour of the form $f(x)\sim
Ce^{bx}$ for $\Re(bx)\to+\infty$. We again start with an ansatz as a
power series
\begin{equation*}
  f(z)=\sum_{n=0}^\infty f_nz^n,
\end{equation*}
from which we derive the recurrence formula
\begin{equation*}
  (n+1)f_{n+1}=(a\lambda^n+b)f_n\text{ for }n\geq0.
\end{equation*}
This gives
\begin{equation}
  \label{eq:f_n}
  f_n=\frac{b^n}{n!}\prod_{k=0}^{n-1}\left(1+\frac ab\lambda^k\right),
\end{equation}
if we set $f_0=1$. In order to simplify notation, we introduce
\begin{equation*}
  Q_\lambda(\alpha)=\prod_{k=0}^\infty\left(1+\alpha\lambda^k\right).
\end{equation*}

This gives
\begin{equation}
  \label{eq:fz1}
  f(z)=Q_\lambda\left(\frac ab\right)
\sum_{n=0}^\infty\frac{(bz)^n}{n!Q_\lambda\left(\frac ab\lambda^n\right)},
\end{equation}
if $\frac ab\lambda^n\neq-1$ for all $n\in\N$. In the case that $\frac
ab\lambda^N$ for some $N\in\N$, the solution degenerates to a
polynomial
\begin{equation*}
  f(z)=Q_\lambda\left(\frac ab\right)
\sum_{n=0}^{N-1}\frac{(bz)^n}{n!Q_\lambda\left(\frac ab\lambda^n\right)},
\end{equation*}

In order to derive an expression for $f(z)$, which allows for
determining its asymptotic behaviour, we recall the well known power
series expansion
\begin{equation}\label{eq:Qlambda}
  \frac1{Q_\lambda(\alpha)}=\sum_{n=0}^\infty(-1)^n
\prod_{k=1}^n\frac1{1-\lambda^k}\alpha^n
\end{equation}
valid for $|\alpha|<1$.

We now choose $N$ as the smallest non-negative integer such that
$|\frac ab\lambda^N|<1$. Then we rewrite \eqref{eq:fz1} as
\begin{equation}
  \label{eq:fz2}
  f(z)=Q_\lambda\left(\frac ab\right)
\sum_{n=0}^{N-1}\frac{(bz)^n}{n!Q_\lambda\left(\frac ab\lambda^n\right)}+
Q_\lambda\left(\frac ab\right)
\sum_{n=N}^\infty\frac{(bz)^n}{n!Q_\lambda\left(\frac ab\lambda^n\right)}.
\end{equation}
We replace $\frac1{Q_\lambda}$ in the second sum by \eqref{eq:Qlambda}
to obtain
\begin{equation}
  \label{eq:fz3}
  f(z)=Q_\lambda\left(\frac ab\right)
\sum_{n=0}^{N-1}\frac{(bz)^n}{n!Q_\lambda\left(\frac
    ab\lambda^n\right)}+
Q_\lambda\left(\frac ab\right)\sum_{m=0}^\infty(-1)^m
\prod_{k=1}^m\frac1{1-\lambda^k}\left(\frac ab\right)^m
\sum_{n=N}^\infty\frac{(b\lambda^mz)^n}{n!},
\end{equation}
which simplifies to
\begin{multline}
  \label{eq:fz4}
  f(z)=Q_\lambda\left(\frac ab\right)
\sum_{n=0}^{N-1}\frac{(bz)^n}{n!Q_\lambda\left(\frac
    ab\lambda^n\right)}\\
+Q_\lambda\left(\frac ab\right)\sum_{m=0}^\infty(-1)^m
\prod_{k=1}^m\frac1{1-\lambda^k}\left(\frac ab\right)^m
\left(e^{b\lambda^mz}-\sum_{n=0}^{N-1}\frac{(b\lambda^mz)^n}{n!}\right).
\end{multline}
\begin{remark}
  Notice that the expansion \eqref{eq:fz4} converges for all values of $a$ and
  $b$ (after a suitable choice of $N$). This is in contrast to a similar series
  expansion given in \cite{Liu_Li2004:properties_analytic}, which converges only
  for $|a|<|b|$.
\end{remark}
\begin{remark}
  In \cite{Liu_Li2004:properties_analytic} a generalization of the classical
  pantograph equation, the multi-pantograph equation
  \begin{equation}\label{eq:multi}
    f'(z)=\sum_{k=1}^\ell a_kf(\lambda_k z)+bf(z)
  \end{equation}
with $1>\lambda_\ell>\lambda_{\ell-1}>\cdots>\lambda_1>0$,
$a_1,\ldots,a_\ell,b\in\R$ is studied. In
\cite[Lemma~3.1]{Liu_Li2004:properties_analytic} a series expansion for the

solution of \eqref{eq:multi} is given, which is shown to converge for
\begin{equation*}
  \sum_{k=1}^\ell|a_k|<|b|.
\end{equation*}
The truncation method presented above can easily be adapted to provide a series
representation similar to \eqref{eq:fz4} for $f$, which does not require this
condition.
\end{remark}

{ \bf Acknowledgements}.  We are grateful to Daniel Berend for
stimulating discussions and  helpful remarks.
G.D.  thankfully acknowledge 
hospitality and support during his research visit  in  Graz,
in May - June 2016, provided by
Institut f\"ur Analysis und Zahlentheorie,
Technische Universit\"at Graz


\begin{thebibliography}{99}

\bibitem{BDM}
 L. V. Bogachev, G.~Derfel, and S. A. Molchanov,
\textit { On bounded continuous solutions of the
archetypal equation with rescaling},
Proc. Royal Soc A{471}, 2015, 1--19.

\bibitem{Bruijn1953:difference-differential}
N.~G. de~Bruijn, \emph{The difference-differential equation
  {$F'(x)=e^{ax+\beta}F(x-1)$}. {I}, {II}}, Nederl. Akad. Wetensch. Proc. Ser.
  A. {\bf 56} = Indagationes Math. \textbf{15} (1953), 449--458, 459--464.

\bibitem{Chandra} S. Chandrasekhar and G. M\"unch, The theory of the
  fluctuations in brightness of the Milky Way,~I.
  \textit{Astrophys.~J.} {\bf 112}, (1950) 380--392.

\bibitem{Derfel1978:asymptotics_class_FDE_Kiev} G. Derfel, On the
  asymptotics of the solutions of a class of functional-differential
  equations, In: \emph{Asymptotic Behavior of the Solutions of
    Functional- Differential Equations} (ed. A.N. Sharkovsky) , 1978,
  58--66, Institute of Math. Ukrainian Acad. Sci Press, Kiev (in
  Russian).

\bibitem{Derfel1995:functional_differential_compressed_arguments}
  G. Derfel, Functional-differential equations with compressed
  arguments and polynomial coefficients: asymptotics of the solutions,
  \textit{J.~Math.\ Anal.\ Appl.}  {\bf 193}, (1995), 671--679.

\bibitem{Derfel1995:FDE_FE_rescaling} G. Derfel,
  Functional-differential and functional equations with rescaling,
  \textit{Operator Theory: Advances and Applications}, {\bf 80},
  (1995), 100--111.

\bibitem {Derfel_Grabner_Vogl2012:laplace_fractal_Poincare} G. Derfel,
  P. J. Grabner, and F. Vogl, Laplace operators on fractals and
  related functional equations (topical review article), \textit
  {J. Phys. A.,} {\bf 46}, (2012), 463001-463034.


\bibitem{Derfel_Iserles1997:pantograph_complex_domain} G. Derfel and
  A. Iserles, The pantograph equation in the complex domain,
  \textit{J.~Math.\ Anal.\ Appl.}  {\bf 213}, (1997), 117--132.


\bibitem{Ferguson} T. S. Ferguson, Lose a dollar or double your
  fortune.  In: \emph{Proceedings of the 6th Berkeley symposium on
    mathematical statistics and probability, vol.~III}
  (eds. L. M. ~Le~Cam \emph{et al.}), 1972, \ 657--666. Berkeley, CA:
  University of California Press.

\bibitem{Gaver} D. Gaver Jr., An absorption probability problem.
  \textit{J.~Math.\ Anal.\ Appl.} {\bf 9}, (1964) 384--393.

\bibitem{Grabner_Steinsky2005:asymptotic_behaviour_poles}
P.~J. Grabner and B.~Steinsky, \emph{Asymptotic behaviour of the poles of a
  special generating function for acyclic digraphs}, Aequationes Math.
  \textbf{70} (2005), 268--278.

\bibitem{HallWake1}
A. Hall and  G. Wake,
A functional differential equation arising in
the modeling of cell growth. \textit{J.~Austral.\ Math.\ Soc.\
Ser.\ B} {\bf 30}, (1989)  424--435.

\bibitem{Iserles} A. Iserles, On the generalized pantograph
  functional-differential equation. \textit{European J.\ Appl.\ Math.}
  (1993), 1-38

\bibitem{Kato_McLeod1971:functional-differential}
T.~Kato and J.~B. McLeod, \emph{The functional-differential equation
  {$y^{\prime} \,(x)=ay(\lambda x)+by(x)$}}, Bull. Amer. Math. Soc. \textbf{77}
  (1971), 891--937.

\bibitem{Langley2000:functional-differential}
J.~K. Langley, \emph{A certain functional-differential equation}, J. Math.
  Anal. Appl. \textbf{244} (2000), no.~2, 564--567.

\bibitem{Liu_Li2004:properties_analytic}
M.~Z. Liu and D.~Li, \emph{Properties of analytic solution and numerical
  solution of multi-pantograph equation}, Appl. Math. Comput. \textbf{155}
  (2004), no.~3, 853--871.

\bibitem{Magnus_Oberhettinger_Soni1966:formulas_theorems_special}
W.~Magnus, F.~Oberhettinger, and R.P. Soni, \emph{Formulas and {T}heorems for
  the {S}pecial {F}unctions of {M}athematical {P}hysics}, Springer, Berlin, New
  York, 1966.

\bibitem{Mahler} K. Mahler, On a special functional
  equation. \textit{J.~London Math.\ Soc.} {\bf 15}, (1940) 115--123.

\bibitem{Morris_Feldstein_Bowen1972:Phragmen_Lindelof_FDE}
  G. R. Morris, A. Feldstein, and E. W. Bowen, The Phragmen --
  Lindel\"of principle and a class of functional - differential
  equations, In: \emph{Ordinary Differential Equations}, 1972,
  513--540, Academic Press, San Diego.

 
\bibitem{OT}
J. R. Ockendon and A. B. Tayler,
The dynamics of a current collection system for an electric locomotive.
\textit{Proc.\ Royal Soc.\ London A}, {\bf 322} (1971), 447--468.

\bibitem{Robinson1973:counting_acyclic_digraphs} R. W. Robinson,
  Counting labeled acyclic digraphs, In: \emph{ New Directions in the
    Theory of Graphs } (ed., F. Harari ), 1973, 239--279 Academic
  Press, New York.

\bibitem{Sokal2012:leading_root_theta_function} A. D. Sokal, The
  leading root of the partial theta function, Advances in Mathematics,
  {\bf 229}, no. 5, (2012), 2603--2621.

\bibitem{Spiridonov}
V. Spiridonov,
Universal superpositions of coherent states and
self-similar potentials. \textit{Phys.\ Rev.\ A} {\bf 52}, (1995)
1909--1935.

\bibitem{Zhang2016:functional-differential}
C.~Zhang, \emph{On the solution to a certain functional differential equation},
  J. Math. Anal. Appl. (2016), to appear, arXiv:1501.02700.

\end{thebibliography}
\end{document}